\documentclass[a4paper,10 pt,conference,times]{ieeeconf}  % Comment this line out if you need a4paper

\IEEEoverridecommandlockouts                              % This command is only needed if 
\usepackage[utf8]{inputenc} 
\usepackage[T1]{fontenc}      
\usepackage[left=0.75in,right=0.514in,bottom=0.75in]{geometry}      
\usepackage{amsmath} 
\usepackage{amsfonts}  
\usepackage{algorithm}
\usepackage{algpseudocode}
\usepackage{units}
\usepackage{graphicx} 
\usepackage{ulem}
\usepackage{nicefrac}

\newtheorem{rk}{Remark}

\newtheorem{ass}{Assumption}
\newtheorem{theo}{Theorem}
\newtheorem{cor}{Corollary}
\newtheorem{lem}{Lemma}
\newtheorem{definition}{Definition}
\newtheorem{prty}{Property}

\newcommand{\Bcal}{\mathcal{B}}

\newcommand{\Ncal}{\mathcal{N}}

\newcommand{\Scal}{\mathcal{S}}

\newcommand{\Xcal}{\mathcal{X}}

\newcommand{\Zcal}{\mathcal{Z}}

\newcommand{\Rset}{\mathbb{R}}

\newcommand{\Nset}{\mathbb{N}}

\newcommand{\Trans}{\scriptscriptstyle\top}

\newcommand{\prox}{\operatorname{prox}}

\DeclareMathOperator*{\minimise}{minimise}

\DeclareMathOperator*{\argm}{argmin}
\DeclareMathOperator*{\infim}{inf}

\title{\LARGE \bf A Parametric Multi-Convex Splitting Technique with Application to Real-Time NMPC}
\author{Jean-Hubert Hours~~~~~~Colin N. Jones
\thanks{ Jean-Hubert Hours and Colin N. Jones are with the Laboratoire d'Automatique,~\'Ecole Polytechnique F\'ed\'erale de Lausanne,~Switzerland. 
        {\tt\small \{jean-hubert.hours, colin.jones\}@epfl.ch}}%
}% <-this % stops a space

\begin{document}
\maketitle
\thispagestyle{empty}
\pagestyle{empty}

\begin{abstract}
A novel splitting scheme to solve parametric multi-convex programs is presented. It consists of a fixed number of proximal alternating minimisations and a dual update per time step, which makes it attractive in a real-time NMPC framework and for distributed computing environments. Assuming that the parametric program is semi-algebraic and that its KKT points are strongly regular, a contraction estimate is derived and it is proven that the sub-optimality error remains stable if two key parameters are tuned properly. Efficacy of the method is demonstrated by solving a bilinear NMPC problem to control a DC motor.
\end{abstract}
%%%%%%%%%%%%%%%%%%%%%%%%%%%%%%%%%%%%%%%%%%%%%%%%%%%%%%%%%%%%%%%%%%%%%%%%%%%%%%%%%%%%%%%%%%%%%%%%%%%%%%%%%%%%%%%%%%%%%%%%%%%%%
\section{Introduction}
\label{sec:intro}
\looseness-1The applicability of NMPC to fast and complex dynamics is hampered by the fact that a nonlinear program (NLP), which is generally non-convex, is to be solved at every sampling time. Solving an NLP to full accuracy is not tractable when the system's sampling frequency is high, which is the case for many mechanical or electrical systems. This difficulty is enhanced when dealing with distributed systems, as they typically lead to large-scale NLPs. Several techniques have been proposed in order to improve the computational efficacy of NMPC schemes by avoiding solving with more accuracy than needed. Most of them rely on the parametric nature of the NLP \cite{zav2009,zav2010,diehl2005}, which has a fixed structure with a time-dependent state or noise estimate. All existing approaches to real-time NMPC are based on Newton type methods, which benefit from local quadratic convergence, but are not easily applicable in a distributed context. Moreover, when inequality constraints are present, a quadratic program (QP) is solved, which is generally performed via active-set methods so as to fully exploit the parametric nature of the NMPC problem by warm-starting \cite{diehl2008}. Unfortunately, the number of iterations required by active-set strategies is hard to predict and thus certification is a difficult problem.\\
\looseness-1In this paper, a parametric optimisation scheme based on augmented Lagrangian \cite{conn1996,bert1982} is proposed. In an NMPC context, such an alternative has already been explored in \cite{zav2010}, which has shown that augmented Lagrangian methods have a good potential for scalable optimisation \cite{zav2013}, partly because they allow one to apply iterative linear algebra and to detect active-sets changes efficiently. In \cite{zav2010}, the theoretical analysis relies on the fact that the primal quadratic program is solved to a given accuracy and the influence of the number of iterations of the suggested projected successive over-relaxation (PSOR) method on the sub-optimality error is not examined. Moreover, the efficacy of the proposed algorithm strongly relies on the fact that the current iterate is close to the optimal solution, as this is required to guarantee convexity of the quadratic program. Finally, due to the dual update, the tracking error is only first-order in the parameter difference. Therefore, the augmented Lagrangian approach may not be very competitive as a fast local method for NMPC, compared to Newton strategies. Yet it is an interesting direction for parallel computing environments or in a distributed NMPC context, assuming that one is able to decompose the evaluation of the primal iterates, as shown in \cite{boyd2010} for convex problems.\\   
\looseness-1The central idea of our algorithm is to apply a truncated version of the proximal alternating minimisation method in \cite{att2010} so as to solve the primal (non-convex) problem approximately. Alternating minimisation strategies are known to lead to `easily' solvable sub-problems, which can be parallelised under some assumptions on the coupling, and are well-suited to distributed computing platforms \cite{bert1997}. At the expense of a few assumptions on the parametric program, we provide an analysis of the stability of the sub-optimality error, in which the parameter difference, the number of primal iterations and the penalty parameter are related. In particular, we give new insights on how the penalty parameter and the number of primal iterations should be tuned in order to ensure boundedness of the tracking error, which is a key point in a real-time NMPC context. In the end, the proposed framework can also address more general problem formulations than \cite{zav2010}, where the PSOR strategy is restricted to quadratic objectives subject to non-negativity constraints.\\
In Section \ref{sec:multi_conv_prg}, the parametric optimisation scheme is presented. In Section \ref{sec:theo_tools}, some key theoretical tools such as Robinson's strong regularity and the Kurdyka-Lojasiewicz property are introduced. Then, in Section \ref{sec:stab_ana}, conditions ensuring stability of the tracking error are derived. In Section \ref{sec:app_rt_nmpc}, the applicability of the parametric optimisation scheme to real-time NMPC is investigated along with basic computational aspects. Finally, the conditions derived in Section \ref{sec:stab_ana} are verified on a numerical example, which consists in controlling the speed of a DC motor to track a piecewise constant reference. An analysis of the evolution of the tracking error as a function of the sampling period for a \textit{fixed computational power} is also presented. 
%%%%%%%%%%%%%%%%%%%%%%%%%%%%%%%%%%%%%%%%%%%%%%%%%%%%%%%%%%%%%%%%%%%%%%%%%%%%%%%%%%%%%%%%%%%%%%%%%%%%%%%%%%%%%%%%%%%%%%%%%%%%%
\section{Background definitions}
\label{sec:not}
\begin{definition}[Critical point]
\label{def:crit_pt}
Let $f$ be a proper lower semicontinuous function. A necessary condition for $x^\ast$ to be a minimiser of $f$ is that 
\begin{align}
\label{eq:crit_pt_def}
0\in{\partial}f(x^\ast)\enspace,
\end{align}
where ${\partial}f(x^\ast)$ is the sub-differential of $f$ at $x^\ast$ \cite{rock2009}. Points satisfying \eqref{eq:crit_pt_def} are called \textit{critical points}. 
\end{definition}
\begin{definition}[Normal cone to a convex set]
Let ${\Omega}$ be a convex set in ${{\Rset}^{n}}$ and ${\bar{x}\in{\Omega}}$. The \textit{normal cone} to ${\Omega}$ at $\bar{x}$ is the set
\begin{align}
{\Ncal}_{\Omega}(\bar{x}):=\left\{v\in\Rset^n~\Big|~\forall{}x\in\Omega,~v^{\Trans}(x-\bar{x})\leq{}0\right\}\enspace.
\end{align}
\end{definition}
The indicator function of a closed subset ${\Omega}$ of ${\Rset^n}$ is denoted by ${\iota_\Omega}$ and is defined as
\begin{align}
\iota_\Omega(x)=\begin{cases}0&\mbox{if } x\in\Omega\\
+\infty&\mbox{if }x\notin\Omega\enspace.\end{cases} 
\end{align}
\begin{lem}[Sub-differential of indicator function \cite{rock2009}]
Given a convex set $\Omega$, for all $x\in\Omega$,  
\begin{align}
\partial\iota_\Omega(x)=\Ncal_\Omega(x)\enspace. 
\end{align}
\end{lem}
The distance of a point $x\in\Rset^n$ to a subset $\Sigma$ of $\Rset^n$ is defined by 
\begin{align}
d(x,\Sigma):=\infim_{y\in\Sigma}\big\|x-y\big\|_2\enspace.
\end{align}
A function $h:\big(z_1,\ldots,z_P\big)\mapsto h\big(z_1,\ldots,z_P\big)$ is said to be \textit{multi-convex} if for all $i\in\left\{1,\ldots,P\right\}$, by fixing variables $z_j$ with $j\ne i$, the resulting function is convex in $z_i$. The open ball with center $x$ and radius $r$ is denoted by $\Bcal\left(x,r\right)$.
%%%%%%%%%%%%%%%%%%%%%%%%%%%%%%%%%%%%%%%%%%%%%%%%%%%%%%%%%%%%%%%%%%%%%%%%%%%%%%%%%%%%%%%%%%%%%%%%%%%%%%%%%%%%%%%%%%%%%%%%%%%%%
\section{Solving time-dependent multi-convex parametric programs}
\label{sec:multi_conv_prg}
\subsection{Problem formulation}
\label{subsec:pb_form}
We consider multi-convex parametric programs
\begin{align}
\label{eq:mu_conv_pb}
&\minimise~f(z_1,\ldots,z_P)\\
&\text{s.t.}~~~~g(z_1,\ldots,z_P,s_k)=0\nonumber\\
&~~~~z_i\in\Zcal_i,~\forall i\in\left\{1,\ldots,P\right\}\enspace,\nonumber
\end{align}
where $f$ is multi-convex in $z:=\left(z_1^{\Trans},\ldots,z_P^{\Trans}\right)^{\Trans}\in\Rset^{n_z}$ with $z_i\in\Rset^{n_i}$ and $n_z:=\sum_{i=1}^Pn_i$, $g(\cdot,s_k)$ is a multi-linear function mapping $\Rset^{n_z}$ into $\Rset^m$, the constraint sets $\Zcal_i$ are compact convex and $k$ is a time-index.~The time-dependent parameter $s_k$ is assumed to lie in a subset $\Scal\subset\Rset^p$. Critical points of the parametric nonlinear program \eqref{eq:mu_conv_pb} are denoted by $z^\ast_k$ or $z^\ast(s_k)$ without distinction.
\begin{ass}[Smoothness and semi-algebraicity]
\label{ass:semi_alg}
The function $f$ is twice continuously differentiable and semi-algebraic.
\end{ass}
\begin{rk}
Note that $g(\cdot,s)$ satisfies Assumption \ref{ass:semi_alg}, as it is multi-linear.
\end{rk}
\subsection{A truncated multi-convex splitting scheme}
\label{subsec:desc_algo}
The basic idea of the proposed algorithm is to track time-dependent local optima $z^{\ast}_k$ of \eqref{eq:mu_conv_pb} by approximately computing saddle points of the augmented Lagrangian
\begin{align}
\label{eq:def_al}
L_{\rho}\left(z,\mu,s_k\right):=f\left(z\right)+\left(\mu+\frac{\rho}{2}g\left(z,s_k\right)\right)^{\Trans}g\left(z,s_k\right)
\end{align}
subject to $z\in\Zcal$, where $\Zcal:=\Zcal_1\times\ldots\times\Zcal_P$, $\mu\in\Rset^m$ is a multiplier associated with the equality constraint $g(z,s_k)=0$ and $\rho>0$ is a well-chosen fixed penalty parameter, as explained in the remainder.

\begin{algorithm}[h!]
	\caption{\label{algo:mu_conv_al}Optimality tracking splitting algorithm}
	\begin{algorithmic}
		\State \textbf{Input:} Suboptimal primal-dual solution $\left(\bar{z}_k^{\Trans},\bar{\mu}_k^{\Trans}\right)^{\Trans}$, parameter $s_{k+1}$, augmented Lagrangian~${L_{\rho}\left(\cdot,\bar{\mu}_k,s_{k+1}\right)}$.
		\State $z^{(0)}\leftarrow\bar{z}_k$
		\For{$l=0\ldots M-1$}
		\For{$i=1\ldots P$}
		\State$z_i^{(l+1)}\leftarrow\underset{z_i\in\Zcal_i}{\argm}~L_{\rho}\Big(z_1^{(l+1)},\ldots,z_{i-1}^{(l+1)},z_i,$\par
							~~~~~~~~~~~~~~~~~~~~~~~~~~~~~~$z_{i+1}^{(l)},\ldots,z_P^{(l)},\bar{\mu}_k,s_{k+1}\Big)$\par
				    			~~~~~~~~~~~~~~~~~~~~~~~~~~~~~~$+\displaystyle\frac{\alpha_i}{2}\left\|z_i-z_i^{(l)}\right\|^2_2$
		\EndFor
		\EndFor
		\State $\bar{z}_{k+1}\leftarrow z^{(M)}$ ; $\bar{\mu}_{k+1}\leftarrow\bar{\mu}_k+{\rho}g\left(\bar{z}_{k+1},s_{k+1}\right)$  	
		\end{algorithmic}
\end{algorithm}
In Algorithm \ref{algo:mu_conv_al} below, the coefficients $\alpha_i>0$ are regularisation parameters, which are chosen to be very small in practice. Algorithm \ref{algo:mu_conv_al} builds a suboptimal solution $\bar{z}_{k+1}$ by applying $M$ iterations of the proximal alternating minimisation method proposed in \cite{att2010}, to evaluate the primal iterates approximately. The dual variable $\mu$ is then updated in a (non-smooth) gradient ascent fashion.
\begin{rk}
Note that each of the subproblems in Algorithm \ref{algo:mu_conv_al} is uniquely solvable, as 
\begin{align}
L_\rho\big(z_1^{(l+1)},\ldots,z_{i-1}^{(l+1)},z_i,z_{i+1}^{(l)},\ldots,z_P^{(l)},&\bar{\mu}_k,s_{k+1}\big)\nonumber\\
&+\frac{\alpha_i}{2}\left\|z_i-z_i^{(l)}\right\|_2^2
\end{align}
is strongly convex and $\Zcal_i$ is convex. 
\end{rk}
%%%%%%%%%%%%%%%%%%%%%%%%%%%%%%%%%%%%%%%%%%%%%%%%%%%%%%%%%%%%%%%%%%%%%%%%%%%%%%%%%%%%%%%%%%%%%%%%%%%%%%%%%%%%%%%%%%%%%%%%%%%%%
\section{Theoretical tools}
\label{sec:theo_tools}
In order to analyse the truncated augmented Lagrangian scheme, we use the concept of generalised equation, which has been introduced in real-time NMPC by \cite{zav2010}.~The stability analysis of the sub-optimality error is also based on the convergence rate of the proximal Gauss-Seidel method in Algorithm \ref{algo:mu_conv_al}.
\subsection{Parametric generalised equations}
\label{subsec:prm_ge}
Critical points $w^{\ast}\left(s_k\right)$ of the parametric nonlinear program \eqref{eq:mu_conv_pb} satisfy the generalised equation 
\begin{align}
\label{eq:ge_1}
0\in F\left(w,s_k\right)+\Ncal_{\Zcal\times\Rset^m}\left(w\right),
\end{align}
where 
\begin{align}
F\left(w,s_k\right):=\begin{bmatrix}
	\nabla_zf\left(z\right)+\nabla_zg\left(z,s_k\right)^{\Trans}\mu\\
	g\left(z,s_k\right)
\end{bmatrix}\enspace,
\end{align}
and $w=\left(z^{\Trans},\mu^{\Trans}\right)^{\Trans}$.\\     
A central concept of our analysis is the \textit{strong regularity} of the generalised equation \eqref{eq:ge_1}.~As addressed in the sequel, strong regularity provides a measure of how close two time-dependent parameters need to be in order to guarantee recursive stability of the sub-optimality error.
\begin{definition}[Strong regularity, \cite{robin1980}]
Given a closed convex set $C$ in $\Rset^n$ and a differentiable mapping $F:\Rset^n\rightarrow\Rset^n$, a generalised equation $0\in F(x)+\Ncal_C\big(x\big)$ is said to be strongly regular at a solution $x^\ast\in C$ if there exists radii $\eta>0$ and $\kappa>0$ such that for all $r\in\Bcal\big(0,\eta\big)$, there exists a unique $x\in\Bcal\big(x^\ast,\kappa\big)$ such that
\begin{align}
r\in F(x^\ast)+{\nabla}F(x^\ast)(x-x^\ast)+\Ncal_C\big(x\big)
\end{align}
and the inverse mapping from $\Bcal\big(0,\eta\big)$ to $\Bcal\big(x^\ast,\kappa\big)$ is Lipschitz continuous. 
\end{definition}
\begin{ass}[Strong regularity of \eqref{eq:ge_1}]
\label{ass:str_reg}
For all time instants $k$ and associated parameters $s_k\in\Scal$, the generalised equation \eqref{eq:ge_1} is strongly regular at a solution $w^{\ast}\left(s_k\right)$ in the sense of \cite{robin1980}. 
\end{ass}
\begin{rk}
Strong regularity of a solution $w^{\ast}\left(s_k\right)$ to \eqref{eq:ge_1} is guaranteed by the \textit{strong second-order sufficient optimality condition} and the standard linear independence constraints qualification \cite{robin1980}. By strengthening the usual second-order sufficient optimality condition, strong regularity does not require strict complementarity slackness to be satisfied, contrary to Fiacco's theorem \cite{fiac1976}.
\end{rk}
From Assumption \ref{ass:str_reg}, the following Lemma can be proven \cite{robin1980}, guaranteeing local Lipschitz continuity of the primal-dual solution to \eqref{eq:ge_1}. 
\begin{lem}[Theorem $2.1$ in \cite{robin1980}]
\label{lem:str_reg}
There exists radii $\delta_A>0$ and $r_A>0$ such that for all $k\in\Nset$, for all $s\in\Bcal\left(s_k,r_A\right)$, there exists a unique 
$w^{\ast}\left(s\right)\in\Bcal\left(w^{\ast}_k,\delta_A\right)$ such that 
\begin{align}
0\in F(w^{\ast}(s),s)+\Ncal_{\Zcal\times\Rset^m}(w^{\ast}(s))
\end{align}
and for all $s,s'\in\Bcal(s_k,r_A)$,
\begin{align}
\left\|w^{\ast}(s)-w^{\ast}(s')\right\|_2\leq\lambda_A\left\|F\left(w^{\ast}(s'),s\right)-F\left(w^{\ast}(s'),s'\right)\right\|_2,
\end{align}
where $\lambda_A>0$ is a Lipschitz constant associated with \eqref{eq:ge_1}.
\end{lem}
\begin{rk}
Without loss of generality, the radii $\delta_A$ and $r_A$ are assumed not to depend on the parameter $s_k$. 
\end{rk}
\begin{ass}
\label{ass:glob_lip_F}
There exists $\lambda_F>0$ such that for all $w\in\Zcal\times\Rset^m$,
\begin{align}
\label{eq:glob_lip_F}
\forall s,s'\in\Scal, \left\|F\left(w,s\right)-F\left(w,s'\right)\right\|_2\leq\lambda_F\left\|s-s'\right\|_2\enspace.
\end{align}
\end{ass}
Such an assumption is valid if, for instance, the parameter $s$ enters the equality constraint $g(z,s)=0$ linearly. In general, Assumption \ref{ass:glob_lip_F} could be replaced with a local Lipschitz continuity property, yet this would make the subsequent analysis dependent on the time instant $k$.  
\subsection{Kurdyka-Lojasiewicz property and convergence rate}
\label{subsec:kl_cv_rate}
The convergence properties of the proximal Gauss-Seidel scheme of Algorithm \ref{algo:mu_conv_al} have been analysed in the case of two alternations \cite{att2010} under fairly general assumptions, the main one being the Kurdyka-Lojasiewicz (KL) property. 
\begin{prty}[KL property]
\label{prop:kl}
A lower semi-continuous function $f$ satisfies the KL property at a point $x^\ast$ in its domain if there exists a neighbourhood $U$ of $x^\ast$, $\eta\in\left(0,+\infty\right]$ and $\phi:\left[0,\eta\right)\rightarrow\Rset_+$ such that $\phi(0)=0$, $\phi$ is $C^1$ on $\left(0,\eta\right)$ with $\phi'>0$ and 
\begin{align}
\phi'\left(f(x)-f(x^\ast)\right)d\left(0,\partial f(x)\right)\geq1\enspace,
\end{align}
for all $x\in U\cap\left\{f\left(x^\ast\right)<f\left(x\right)<f\left(x^\ast\right)+\eta\right\}$.
\end{prty}
Given a semi-algebraic function $L:\Rset^n\rightarrow\Rset\cup\left\{+\infty\right\}$, it can actually be shown that $L$ satisfies the KL property at a given critical point $x^\ast$ with $\phi(t)=ct^{1-\theta}$ \cite{bolte2007}, that is there exists $\delta>0$, $c>0$ and $\theta\in\left[0,1\right)$ such that for all $x\in\Bcal\left(x^\ast,\delta\right)\cap\left\{x\in\Rset^n~\big|~L(x)>L(x^{\ast})\right\}$,  
\begin{align}
\label{eq:kl}
d\left(0,\partial L(x)\right)\geq c\left(L(x)-L(x^{\ast})\right)^\theta\enspace,
\end{align}
where $\theta$ is taken as the smallest possible exponent satisfying \eqref{eq:kl}. The parameter $\theta$ can be seen as a shape parameter of the graph of $L$ around a critical point $x^\ast$. When $\theta$ is close to $0$, the graph is sharp at $x^\ast$. When $\theta$ is close to $1$, the graph is flat around $x^\ast$. 
\begin{ass}
\label{ass:uni_kl}
The augmented Lagrangian \eqref{eq:def_al} satisfies the KL property for all~$\mu\in\Rset^m$ and~$s\in\Scal$ with Lojasiewicz exponents~$\theta\left(\mu,s\right)\in\left(\nicefrac{1}{2},1\right)$ and radius $\delta>0$ at its critical points.~The exponents~$\theta\left(\mu,s\right)$ can be upper bounded by~$\hat{\theta}\in\left(\nicefrac{1}{2},1\right)$. 
\end{ass} 
\begin{rk}
Such an assumption is not unreasonable.~It can be proven that for real analytic functions, the exponent~$\theta$ lies within~$\left[\nicefrac[]{1}{2},1\right)$ \cite{loja1963}.~Moreover, for multivariate polynomials of degree higher than two, such as~$L_{\rho}\left(\cdot,\mu,s\right)$, an upper bound on $\theta$ can be computed, which depends only on the number of variables and the degree~\cite{acunto2005}. In many cases, the radius $\delta$ is large. For instance, in the case of strongly convex functions, $\delta=+\infty$.    
\end{rk}
The following Lemma is a trivial extension of the result of \cite{att2010} to the multi-stage case.
\begin{lem}[Theorem $3.2$ in \cite{att2010}]
\label{lem:bcd_cv}
Assuming that $M=\infty$, the sequence $\left\{z^{(l)}\right\}$ generated by the inner loop of Algorithm \ref{algo:mu_conv_al} converges to a critical point $z^\infty\left(\bar{\mu}_k,s_{k+1}\right)$ of $L_\rho\left(\cdot,\bar{\mu}_k,s_{k+1}\right)+\iota_\Zcal\left(\cdot\right)$.
\end{lem}
This convergence result comes with a local sub-linear R-convergence rate estimate.  
\begin{lem}[Local R-convergence rate estimate]
~\\
\label{lem:bcd_cv_rate}
There exists a constant $C>0$ such that, assuming $\bar{z}_k\in\Bcal\left(0,\delta\right)$,  
\begin{align}
\label{eq:cv_rate}
&\left\|\bar{z}_{k+1}-z^\infty\left(\bar{\mu}_k,s_{k+1}\right)\right\|_2\leq\nonumber\\
&~~~~~~~~~~~~~~~~~~~~~~~~~~~~CM^{-\psi\left(\hat{\theta}\right)}{\left\|\bar{z}_k-z^\infty\left(\bar{\mu}_k,s_{k+1}\right)\right\|_2}\enspace,
\end{align}
where, given $\theta\in\left(\nicefrac[]{1}{2},1\right)$,\begin{align}
\psi\left(\theta\right):=\frac{1-\theta}{2\theta-1}\enspace.
\end{align}
\end{lem}
\begin{proof}
From \cite{att2009}, as  the Lojasiewicz exponent $\theta\left(\bar{\mu}_k,s_{k+1}\right)$ associated with $z^\infty\left(\bar{\mu}_k,s_{k+1}\right)$ lies in $\left(\nicefrac[]{1}{2},1\right)$, by Assumption \ref{ass:uni_kl}, and $\bar{z}_k\in\Bcal\left(0,\delta\right)$, it can be shown that, given $\bar{\mu}_k\in\Rset^m$ and $s_{k+1}\in\Scal$, there exists $C\left(\bar{\mu}_k,s_{k+1}\right)>0$ such that 
\begin{align}
\left\|\bar{z}_{k+1}-z^\infty\left(\bar{\mu}_k,s_{k+1}\right)\right\|_2\leq C\left(\bar{\mu}_k,s_{k+1}\right)M^{-\psi\left(\theta\left(\bar{\mu}_k,s_{k+1}\right)\right)}\enspace.
\end{align}
Note that $\theta\mapsto M^{-\psi\left(\theta\right)}$ is strictly increasing on $\left(\nicefrac[]{1}{2},1\right)$.
Hence, from Assumption \ref{ass:uni_kl}, 
\begin{align}
M^{-\psi\left(\theta\left(\bar{\mu}_k,s_{k+1}\right)\right)}\leq M^{-\psi\left(\hat{\theta}\right)}\enspace.
\end{align}
Clearly, as $\bar{z}_k$ is the suboptimal primal solution of \eqref{eq:mu_conv_pb} at time $k$, there exists $\kappa>0$ such that for all $k\geq 0$, 
\begin{align} 
\left\|\bar{z}_k-z^\infty\left(\bar{\mu}_k,s_{k+1}\right)\right\|_2\geq\kappa\enspace.
\end{align}
Hence there exists $C'\left(\bar{\mu}_k,s_{k+1}\right)>0$ such that 
\begin{align}
&\left\|\bar{z}_{k+1}-z^\infty(\bar{\mu}_k,s_{k+1})\right\|_2\leq\nonumber\\ 
&~~~~~~~~~~~~~~~C'\left(\bar{\mu}_k,s_{k+1}\right)M^{-\psi\left(\hat{\theta}\right)}\left\|\bar{z}_k-z^\infty(\bar{\mu}_k,s_{k+1})\right\|_2\enspace.
\end{align}
Without loss of generality, one can assume that the constants $C'\left(\bar{\mu}_k,s_{k+1}\right)$ are upper bounded, which yields \eqref{eq:cv_rate}.
\end{proof}
\begin{rk}
Note that the R-convergence rate of Lemma \ref{lem:bcd_cv_rate} shows that convergence of the multi-convex alternations is theoretically quite slow. Yet, the algorithm is observed to be quite efficient in practice, as shown in Section \ref{sec:num_ex}. We insist on the fact that it is an upper bound, which is used for theoretical purpose only. 
\end{rk}
%%%%%%%%%%%%%%%%%%%%%%%%%%%%%%%%%%%%%%%%%%%%%%%%%%%%%%%%%%%%%%%%%%%%%%%%%%%%%%%%%%%%%%%%%%%%%%%%%%%%%%%%%%%%%%%%%%%%%%%%%%%%%
\section{Contraction analysis of the optimality tracking algorithm}
\label{sec:stab_ana}
As Algorithm \ref{algo:mu_conv_al} is a truncated scheme applied online for varying values of the parameters $s\in\Scal$, a natural question is: under which conditions does the sub-optimal primal-dual solution converge to a solution of \eqref{eq:mu_conv_pb} as the parameter $s$ varies ? More precisely, is it possible to ensure that for all $k\geq0$, $\left\|\bar{w}_{k+1}-w^\ast_{k+1}\right\|_2\leq\alpha\left\|\bar{w}_k-w^\ast_k\right\|_2$, where $\alpha$ is a constant in $\left(0,1\right)$ ?\\
In the sequel, we show that if $\rho$ and $M$ are carefully chosen, such a contraction property is satisfied in a weak sense, and the error sequence $\left\|\bar{w}_k-w^\ast_k\right\|_2$ remains bounded with $k$, assuming that the parameter difference $\left\|s_{k+1}-s_k\right\|_2$ is small enough.
\subsection{Existence and uniqueness of critical points}
\label{subsec:ex_uni}
Given a critical point $w^\ast_k$ of problem \eqref{eq:mu_conv_pb}, strong regularity of \eqref{eq:ge_1} implies that a critical point of \eqref{eq:mu_conv_pb} exists for $s=s_{k+1}$ and is unique in a neighbourhood of  $w^\ast_k$, assuming that $s_{k+1}$ is in a well-chosen neighbourhood of $s_k$. 
\begin{ass}
\label{ass:ngh_prm}
For all $k\in\Nset$, $\left\|s_{k+1}-s_k\right\|_2\leq r_A$.
\end{ass}
\begin{lem}
For all $k\in\Nset$ and $s_k\in\Scal$, given $w^\ast_k$ satisfying \eqref{eq:ge_1}, there exists a unique $w^\ast_{k+1}\in\Bcal\left(w^\ast_k,\delta_A\right)$ such that 
\begin{align}
0\in F\left(w^\ast_{k+1},s_{k+1}\right)+\Ncal_{\Zcal\times\Rset^m}\left(w^\ast_{k+1}\right)\enspace.
\end{align}
\end{lem}
\begin{proof}
Immediate from Assumption \ref{ass:ngh_prm} and strong regularity of \eqref{eq:ge_1}.
\end{proof}
\subsection{An auxiliary generalised equation}
\label{subsec:aux_ge}
In Algorithm \ref{algo:mu_conv_al}, the proximal alternating loop, warm-started at $\bar{z}_k$, converges to $z^\infty\left(\bar{\mu}_k,s_{k+1}\right)$, which is a critical point of $L_\rho\left(\cdot,\bar{\mu}_k,s_{k+1}\right)+\iota_\Zcal\left(\cdot\right)$, by Lemma \ref{lem:bcd_cv}. The following generalised equation characterises critical points of the augmented Lagrangian function $L_\rho\left(\cdot,\bar{\mu},s\right)+\iota_\Zcal\left(\cdot\right)$ in a primal-dual manner, which is helpful in our analysis:
\begin{align}
\label{eq:ge_2}
0\in G_\rho\left(w,d_\rho\left(\bar{\mu}\right),s\right)+\Ncal_{\Zcal\times\Rset^m}\left(w\right)\enspace,
\end{align}
where $d_\rho\left(\bar{\mu}\right):=\left(\bar{\mu}-\mu^\ast_k\right)/\rho$ and 
\begin{align}
G_\rho\left(w,d_\rho\left(\bar{\mu}\right),s\right):=\begin{bmatrix}
\nabla_zf\left(z\right)+\nabla_zg\left(z,s\right)^{\Trans}\mu\\
g\left(z,s\right)+d_\rho\left(\bar{\mu}\right)+\displaystyle\frac{\mu^{\ast}_k-\mu}{\rho}
\end{bmatrix}\enspace.
\end{align}
In the sequel, a primal-dual point satisfying \eqref{eq:ge_2} is denoted by $w^{\ast}\left(d_\rho\left(\bar{\mu}\right),s\right)$ or $w^{\ast}\left(\bar{\mu},s\right)$ without distinction. 
\begin{lem}
Let $\bar{\mu}\in\Rset^m$, $\rho>0$ and $s\in\Scal$. The primal point $z^\ast(\bar{\mu},s)$ is a critical point of $L_{\rho}(\cdot,\bar{\mu},s)+\iota_\Zcal(\cdot)$ if and only if the primal-dual point 
\begin{align}
w^\ast(\bar{\mu},s)=\begin{pmatrix}
z^\ast\left(\bar{\mu},s\right)\\
\bar{\mu}_k+{\rho}g\left(z^\ast\left(\bar{\mu},s\right),s\right)
\end{pmatrix} 
\end{align}
is a solution of \eqref{eq:ge_2}.
\end{lem}
\begin{proof}
The necessary condition is clear. To prove the sufficient condition, assume that $w^\ast\left(\bar{\mu},s\right)=\left(z^\ast\left(\bar{\mu},s\right)^{\Trans},\mu^\ast\left(\bar{\mu},s\right)^{\Trans}\right)^{\Trans}$ satisfies \eqref{eq:ge_2}. The second half of \eqref{eq:ge_2} implies that $\mu^\ast\left(\bar{\mu},s\right)=\bar{\mu}+\rho g\left(z^\ast\left(\bar{\mu},s\right),s\right)$. Putting this expression in the first part of \eqref{eq:ge_2}, this implies that $z^\ast\left(\bar{\mu},s\right)$ is a critical point of $L_\rho\left(\cdot,\bar{\mu},s\right)+\iota_\Zcal\left(\cdot\right)$.
\end{proof}
As $z^\infty\left(\bar{\mu}_k,s_{k+1}\right)$ is a critical point of $L_\rho\left(\cdot,\bar{\mu}_k,s_{k+1}\right)+\iota_\Zcal\left(\cdot\right)$, one can define 
\begin{align}
\label{eq:def_w_inf}
w^\infty\left(d_\rho(\bar{\mu}_k),s_{k+1}\right):=\begin{pmatrix}
z^\infty\left(\bar{\mu}_k,s_{k+1}\right)\\
\bar{\mu}_k+\rho g\left(z^\infty(\bar{\mu}_k,s_{k+1}),s_{k+1}\right)
\end{pmatrix}\enspace,
\end{align}
which satisfies \eqref{eq:ge_2}.
Note that the generalised equation \eqref{eq:ge_2} is parametric in $s$ and $d_\rho(\cdot)$, which represents the normalised distance between the sub-optimal dual and the optimal dual parameters.
Assuming that the penalty parameter $\rho$ is well-chosen, the generalised equation \eqref{eq:ge_2} can be proven to be strongly regular at a given solution.
\begin{lem}[Strong regularity of \eqref{eq:ge_2}]
There exists $\tilde{\rho}>0$ such that for all $\rho>\tilde{\rho}$ and $k\in\Nset$, \eqref{eq:ge_2} is strongly regular at $w^\ast_k=w^\ast\left(0,s_k\right)$. 
\end{lem}
\begin{proof}
This follows from the reduction procedure described in \cite{robin1980}, the arguments developed in Proposition $2.4$ in \cite{bert1982} and strong regularity of \eqref{eq:ge_1} for all $k\in\Nset$.
\end{proof}
\begin{ass}
\label{ass:rho}
The penalty parameter satisfies $\rho>\tilde{\rho}$.
\end{ass}
From the strong regularity of \eqref{eq:ge_2} at $w^\ast_k$, using Theorem $2.1$ in \cite{robin1980}, one obtains the following local Lipschitz property of a solution $w\left(\cdot\right)$ to \eqref{eq:ge_2}. 
\begin{lem}
\label{lem:str_reg_2}
There exists radii $\delta_B>0$, $r_B>0$ and $q_B>0$ such that for all $k\in\Nset$,
\begin{align}
\forall d\in\Bcal\left(0,q_B\right),&\forall s\in\Bcal\left(s_k,r_B\right),\exists!w^{\ast}(d,s)\in\Bcal\left(w^{\ast}_k,\delta_B\right),\nonumber\\
&0\in G_\rho(w^{\ast}(d,s),d,s)+\Ncal_{\Zcal\times\Rset^m}(w^{\ast}(d,s))
\end{align}
and for all $d,d'\in\Bcal\left(0,q_B\right)$ and all $s,s'\in\Bcal\left(s_k,r_B\right)$,
\begin{align}
&\left\|w^{\ast}(d,s)-w^{\ast}(d',s')\right\|_2\leq\nonumber\\
&~~~~~~~~\lambda_B\left\|G_\rho\left(w^\ast(d',s'),d,s\right)-G_\rho\left(w^\ast(d',s'),d',s'\right)\right\|_2\enspace,
\end{align}
where $\lambda_B>0$ is a Lipschitz constant associated with \eqref{eq:ge_2}.
\end{lem}
Note that, given $w\in\Zcal\times\Rset^m$, $d,d'\in\Rset^m$ and $s,s'\in\Scal$, one can write
\begin{align}
G_\rho\left(w,d,s\right)-G_\rho\left(w,d',s'\right)=&F(w,s)-F(w,s')\nonumber\\
&+\begin{bmatrix}
0\\
d-d'
\end{bmatrix}\enspace,
\end{align}
which, from Assumption \ref{ass:glob_lip_F}, implies the following Lemma.
\begin{lem}
\label{lem:glob_lip_G}
There exists $\lambda_G>0$ such that for all $w\in\Zcal\times\Rset^m$, for all $d,d'\in\Rset^m$ and all $s,s'\in\Rset^m$, 
\begin{align}
\label{eq:glob_lip_G}
\left\|G_\rho\left(w,d,s\right)-G_\rho\left(w,d',s'\right)\right\|_2\leq\lambda_G\left\|\begin{pmatrix}
d\\ s
\end{pmatrix}-\begin{pmatrix}
d'\\ s'
\end{pmatrix}\right\|_2\enspace.
\end{align}
\end{lem}
\begin{proof}
%\begin{align}
%\big\|G_\rho\big(w,d,s\big)-G_\rho\big(w,d',s'\big)\big\|_2^2=&\big\|F(w,s)-F(w,s')\big\|_2^2\nonumber\\
%															&+\big\|d-d'\big\|_2^2\nonumber\\
%															+2\big(d-&d'\big)^{\Trans}\big(g(z,s)-g(z,s')\big)\nonumber\\
%\leq&L_F^2\big\|s-s'\big\|_2^2+\big\|d-d'\big\|_2^2\nonumber\\
%&+2L_F\big\|d-d'\big\|_2\big\|s-s'\big\|_2\nonumber\\
%\leq\big(\max\big\{L_F^2,1&\big\}+L_F\big)\left\|\begin{pmatrix}s-s'\\ d-d'\end{pmatrix}\right\|_2^2\enspace.
%\end{align}
After straightforward calculations, one obtains the Lipschitz property with  
\begin{align}
\lambda_G:=\sqrt{\max\big\{\lambda_F^2,1\big\}+\lambda_F}\enspace.
\end{align}
\end{proof}
\subsection{Contraction estimate}
\label{subsec:stab_ana}
This paragraph is the core of the paper and is devoted to proving that under some conditions, which are made explicit in the sequel, the optimality tracking error $\left\|\bar{w}_k-w^\ast_k\right\|_2$ of Algorithm \ref{algo:mu_conv_al} decreases as the parameter $s$ varies slowly.\\
First, note that given a sub-optimal primal-dual solution $\bar{w}_{k+1}$ and a critical point $w^\ast_{k+1}$, 
\begin{align}
\label{eq:stb_ineq}
\left\|\bar{w}_{k+1}-w^{\ast}_{k+1}\right\|_2\leq&\left\|\bar{w}_{k+1}-w^\infty\left(d_\rho\left(\bar{\mu}_k\right),s_{k+1}\right)\right\|_2\nonumber\\
&+\left\|w^\infty\left(d_\rho\left(\bar{\mu}_k\right),s_{k+1}\right)-w^\ast_{k+1}\right\|_2\enspace,
\end{align}
where $w^\infty\left(d_\rho(\bar{\mu}_k),s_{k+1}\right)$ has been defined in \eqref{eq:def_w_inf}.
The analysis then consists in bounding the two right hand side terms in \eqref{eq:stb_ineq}, for the first term using strong regularity of \eqref{eq:ge_2} and for the second one using the convergence rate of the primal loop in Algorithm \ref{algo:mu_conv_al}. 
\begin{lem}
\label{lem:fir_term}
If $\left\|s_{k+1}-s_k\right\|_2$ satisfies 
\begin{align}
\label{eq:hypo_ds_1}
\left\|s_{k+1}-s_k\right\|_2<\min\left\{r_B,\displaystyle\frac{q_B\rho}{\lambda_A\lambda_F}\right\}\enspace,
\end{align}
and $\left\|\bar{w}_k-w^\ast_k\right\|_2<q_B\rho$, 
\begin{align}
\left\|w^\infty\left(d_\rho\left(\bar{\mu}_k\right),s_{k+1}\right)-w^\ast_{k+1}\right\|_2\leq&\displaystyle\frac{\lambda_B\lambda_G}{\rho}\big(\left\|\bar{w}_k-w^\ast_k\right\|_2\nonumber\\
&+\lambda_A\lambda_F\left\|s_{k+1}-s_k\right\|_2\big)\enspace.
\end{align}
\end{lem}
\begin{proof}
Note that $w^\ast_{k+1}$ can be written as $w^\ast_{k+1}=w^\ast\left(d_\rho\left(\mu^\ast_{k+1}\right),s_{k+1}\right)$, which is a solution to \eqref{eq:ge_2} at $s_{k+1}$.
\begin{align}
\big\|d_\rho\big(\mu^\ast_{k+1}\big)\big\|_2=\displaystyle\frac{\left\|\mu^\ast_{k+1}-\mu^\ast_k\right\|_2}{\rho}&\leq\displaystyle\frac{\lambda_F\lambda_A}{\rho}\left\|s_{k+1}-s_k\right\|_2\nonumber\\
&<q_B\enspace,
\end{align}
by applying Lemma \ref{lem:str_reg}, Assumption \ref{ass:glob_lip_F} and from hypothesis \eqref{eq:hypo_ds_1}. Moreover,
\begin{align}
\big\|d_\rho(\bar{\mu}_k)\big\|_2=\displaystyle\frac{\big\|\bar{\mu}_k-\mu^\ast_k\big\|_2}{\rho}&\leq\displaystyle\frac{\big\|\bar{w}_k-w^\ast_k\big\|_2}{\rho}\nonumber\\
&<q_B\enspace.
\end{align}
Now, as $\big\|s_{k+1}-s_k\big\|_2<r_B$ one can apply Lemmas \ref{lem:str_reg_2} and \ref{lem:glob_lip_G} to obtain
\begin{align}
\left\|w^\infty\left(\bar{\mu}_k,s_{k+1}\right)-w^\ast_{k+1}\right\|_2&\leq\lambda_B\lambda_G\left\|d_\rho\left(\bar{\mu}_k\right)-d_\rho\left(\mu^\ast_{k+1}\right)\right\|_2\nonumber\\
\leq\displaystyle\frac{\lambda_B\lambda_G}{\rho}&\left(\left\|\bar{\mu}_k-\mu^\ast_k\right\|_2+\left\|\mu^\ast_{k+1}-\mu^\ast_k\right\|_2\right)\nonumber\\
\leq\displaystyle\frac{\lambda_B\lambda_G}{\rho}\big(\big\|\bar{w}_k-&w^\ast_k\big\|_2+\lambda_A\lambda_F\big\|s_{k+1}-s_k\big\|_2\big)\enspace,
\end{align}
by Lemma \ref{lem:str_reg}.
\end{proof}
In the following Lemma, using the convergence rate estimate presented in Section \ref{sec:theo_tools}, we derive a bound on the first summand $\left\|\bar{w}_{k+1}-w^\infty(d_\rho(\bar{\mu}_k),s_{k+1})\right\|_2$.
\begin{lem}
\label{lem:sec_term}
If $\left\|s_{k+1}-s_k\right\|_2<r_B$, $\left\|\bar{w}_k-w^\ast_k\right\|_2<q_B\rho$ and
\begin{align}
%\big(1+\displaystyle\frac{L_GL_B}{\rho}\big)\big\|\bar{w}_k-w^\ast_k\big\|_2+L_GL_B\big\|s_{k+1}-s_k\big\|_2<\delta\enspace,
\big(1+\displaystyle\frac{\lambda_G\lambda_B}{\rho}\big)q_B\rho+\lambda_G\lambda_Br_B<\delta\enspace,
\end{align}
then  
\begin{align}
\label{eq:bnd_sec}
\big\|\bar{w}_{k+1}-w^\infty&\left(d_\rho\left(\bar{\mu}_k\right),s_{k+1}\right)\big\|_2\leq\nonumber\\
&C\left(1+{\rho}\lambda_g\right)M^{-\psi\left(\hat{\theta}\right)}\Big(\lambda_B\lambda_G\left\|s_{k+1}-s_k\right\|_2\nonumber\\
&~~~~~~~~+\big\|\bar{w}_k-w^\ast_k\big\|_2\Big(1+\displaystyle\frac{\lambda_B\lambda_G}{\rho}\Big)\Big)\enspace, 
\end{align}
where $\lambda_g>0$ is the Lipschitz constant of $g(\cdot,s)$ on $\Zcal$ (well-defined as $\Zcal$ is bounded). 
\end{lem}
\begin{proof}
From Algorithm \ref{algo:mu_conv_al}, it follows that
\begin{align}
\big\|\bar{w}_{k+1}-&w^\infty\big(d_\rho\big(\bar{\mu}_k\big),s_{k+1}\big)\big\|_2\leq\nonumber\\
&\left\|\begin{pmatrix}
\bar{z}_{k+1}-z^\infty\left(\bar{\mu}_k,s_{k+1}\right)\\
\rho\left(g\left(\bar{z}_{k+1},s_{k+1}\right)-g\left(z^\infty\left(\bar{\mu}_k,s_{k+1}\right),s_{k+1}\right)\right)
\end{pmatrix}\right\|_2\nonumber\\
&\leq\left(1+{\rho}\lambda_g\right)\left\|\bar{z}_{k+1}-z^\infty\left(\bar{\mu}_k,s_{k+1}\right)\right\|_2\enspace.
\end{align}
In order to apply Lemma \ref{lem:bcd_cv_rate}, one first need to show that $\bar{z}_k$ lies in the ball $\Bcal\big(z^\infty(\bar{\mu}_k,s_{k+1}),\delta\big)$, where $\delta$ is the radius involved in the KL property.
\begin{align}
\label{eq:ineqs_ball}
\big\|\bar{z}_k-z^\infty(\bar{\mu}_k,s_{k+1})\big\|_2\leq&\big\|\bar{z}_k-z^\ast(0,s_k)\big\|_2\nonumber\\
&+\big\|z^\ast(0,s_k)-z^\infty(\bar{\mu}_k,s_{k+1})\big\|_2\nonumber\\
\leq&\big\|\bar{w}_k-w^\ast_k\big\|_2\nonumber\\
+&\lambda_G\lambda_B\big(\left\|d_\rho(\bar{\mu}_k)\right\|_2+\left\|s_{k+1}-s_k\right\|_2\big)\nonumber\\
\leq&\left(1+\displaystyle\frac{\lambda_G\lambda_B}{\rho}\right)\left\|\bar{w}_k-w^\ast_k\right\|_2\nonumber\\
&+\lambda_G\lambda_B\left\|s_{k+1}-s_k\right\|_2<\delta\enspace,
\end{align}
where the second step follows from strong regularity of \eqref{eq:ge_2} at $w^\ast(0,s_k)$ and the hypotheses mentioned above. Thus one can use the R-convergence rate estimate in Lemma \ref{lem:bcd_cv_rate} and apply the inequalities in \eqref{eq:ineqs_ball} to obtain \eqref{eq:bnd_sec}. 
\end{proof}
Gathering the results of Lemmas \ref{lem:fir_term} and \ref{lem:sec_term}, one can formalise the following theorem. 
\begin{theo}[Contraction]
Given a time instant $k$, if the primal-dual error $\left\|\bar{w}_k-w^\ast_k\right\|_2$, the number of primal iterations $M$, the penalty parameter $\rho$ and the parameter difference $\left\|s_{k+1}-s_k\right\|_2$ satisfy
\begin{itemize}
\item $\left\|s_{k+1}-s_k\right\|_2<\min\left\{r_A,r_B,\displaystyle\frac{q_B\rho}{\lambda_A\lambda_F}\right\}\enspace,$
\item $\left\|\bar{w}_k-w^\ast_k\right\|_2<q_B\rho\enspace,$
\item $\rho>\tilde{\rho}\enspace,$
\item\begin{align}
\label{eq:hyp_kl_rad}
\hspace{-0.4cm}\left(1+\displaystyle\frac{\lambda_G\lambda_B}{\rho}\right)\left\|\bar{w}_k-w^\ast_k\right\|_2+\lambda_G\lambda_B\left\|s_{k+1}-s_k\right\|_2<\delta\enspace,
\end{align}  
\end{itemize}
then
\begin{align}
\label{eq:wk_con}
\left\|\bar{w}_{k+1}-w^\ast_{k+1}\right\|_2\leq&\beta_w\left(\rho,M\right)\left\|\bar{w}_k-w^\ast_k\right\|_2\nonumber\\
&~~~~~~~~~~~~+\beta_s\left(\rho,M\right)\left\|s_{k+1}-s_k\right\|_2\enspace,
\end{align} where
\begin{align}
\label{eq:beta_w}
\beta_w\left(\rho,M\right):=&~C\left(1+\rho\lambda_g\right)\left(1+\displaystyle\frac{\lambda_B\lambda_G}{\rho}\right)M^{-\psi\left(\hat{\theta}\right)}+\displaystyle\frac{\lambda_B\lambda_G}{\rho}\enspace,
\end{align} and
\begin{align}
\label{eq:beta_s}
\beta_s\left(\rho,M\right):=&~C\left(1+\rho\lambda_g\right)\lambda_B\lambda_GM^{-\psi\left(\hat{\theta}\right)}+\displaystyle\frac{\lambda_B\lambda_G\lambda_A\lambda_F}{\rho}\enspace.
\end{align}
\end{theo}
\begin{proof}
This is a direct consequence of Lemmas \ref{lem:fir_term} and \ref{lem:sec_term}.
\end{proof}
\begin{rk}
Note that the last hypothesis \eqref{eq:hyp_kl_rad} may be quite restrictive, since $\left\|\bar{w}_k-w^\ast_k\right\|_2$ needs to be small enough for it to be satisfied. However, in many cases the radius $\delta$ is large ($+\infty$ for strongly convex functions).
\end{rk}
In order to ensure stability of the sequence of sub-optimal iterates $\bar{w}_k$, the parameter difference $\left\|s_{k+1}-s_k\right\|_2$ has to be small enough and the coefficient $\beta_w\left(\rho,M\right)$ needs  to be strictly less than $1$. This last requirement is clearly satisfied if $\rho$ is large enough to make $\displaystyle\nicefrac{\lambda_B\lambda_G}{\rho}$ small in \eqref{eq:beta_w}. Yet $\rho$ also appears in $1+\rho\lambda_g$. Hence it needs to be balanced by a large enough number of primal iterations $M$ in order to make the first summand in \eqref{eq:beta_w} small. The same analysis applies to the second coefficient $\beta_s\left(\rho,M\right)$ in order to mitigate the effect of the parameter difference $\left\|s_{k+1}-s_k\right\|_2$.
\begin{cor}[Boundedness of the error sequence]
Assume that $\rho$ and $M$ have been chosen so that $\beta_w\left(\rho,M\right)$ and $\beta_s\left(\rho,M\right)$ are strictly less than $1$, and $\rho>\tilde{\rho}$. Let $r_w>0$ such that $\delta-\big(1+\displaystyle\nicefrac{\lambda_G\lambda_B}{\rho}\big)r_w>0$ and $r_w<q_B\rho$. Let $r_s>0$ such that $r_s<\displaystyle\nicefrac{(1-\beta_w(\rho,M))r_w}{\beta_s(\rho,M)}$.\\
If $\left\|\bar{w}_0-w^\ast_0\right\|_2<r_w$ and for all $k\geq0$,
\begin{align}
\label{eq:hyp_prm_diff}
\left\|s_{k+1}-s_k\right\|_2\leq\min\left\{r_s,r_A,r_B,\displaystyle\frac{q_B\rho}{\lambda_A\lambda_F}\right\}\enspace,
\end{align}
then for all $k\geq0$, the error sequence satisfies 
\begin{align}
\left\|\bar{w}_k-w^\ast_k\right\|_2<r_w\enspace.
\end{align}
\end{cor}
\begin{proof}
The proof proceeds by a straightforward induction. At $k=0$, $\left\|\bar{w}_0-w^\ast_0\right\|_2<r_w$, by assumption. Let $k\geq0$ and assume that $\left\|\bar{w}_k-w^\ast_k\right\|_2<r_w$. As $\left\|s_{k+1}-s_k\right\|_2<r_A$, by applying Lemma \ref{lem:str_reg}, there exists a unique $w^\ast_{k+1}\in\Bcal\left(w^\ast_k,\delta_A\right)$, which satisfies \eqref{eq:ge_1}. As $\left\|s_{k+1}-s_k\right\|_2$ satisfies \eqref{eq:hyp_prm_diff}, $\left\|\bar{w}_k-w^\ast_k\right\|_2<q_B\rho$, $\rho>\tilde{\rho}$ and \eqref{eq:hyp_kl_rad} is satisfied, from the choice of $r_w$ and $r_s$, we have
\begin{align}
\left\|\bar{w}_{k+1}-w^\ast_{k+1}\right\|_2&\leq\beta_w\left(\rho,M\right)\big\|\bar{w}_k-w^\ast_k\big\|_2\nonumber\\
&~~~~~~~~~~~~~~~~+\beta_s\left(\rho,M\right)\left\|s_{k+1}-s_k\right\|_2\nonumber\\
&\leq\beta_w\left(\rho,M\right)r_w+\beta_s\left(\rho,M\right)\left\|s_{k+1}-s_k\right\|_2\nonumber\\
&\leq r_w\enspace,
\end{align}
as $\left\|s_{k+1}-s_k\right\|_2\leq r_s<\displaystyle\nicefrac{(1-\beta_w(\rho,M))r_w}{\beta_s(\rho,M)}$. Note from the choice of $r_w$ and $r_s$, the condition \eqref{eq:hyp_kl_rad} guaranteeing the weak contraction \eqref{eq:wk_con} is also recursively satisfied.
\end{proof}
In the remainder, we show that such a tuning of $\rho$, $M$ and $\left\|s_{k+1}-s_k\right\|_2$, which ensures stability of the error sequence, is actually possible on a realistic numerical example and that good tracking performance can be achieved.
%%%%%%%%%%%%%%%%%%%%%%%%%%%%%%%%%%%%%%%%%%%%%%%%%%%%%%%%%%%%%%%%%%%%%%%%%%%%%%%%%%%%%%%%%%%%%%%%%%%%%%%%%%%%%%%%%%%%%%%%%%%%%
\section{Application to real-time NMPC}
\label{sec:app_rt_nmpc}
\subsection{Computational aspects}
\label{subsec:comp_asp}
Algorithm \ref{algo:mu_conv_al} allows one to address a more general class of problems than in \cite{zav2010}, where the QP sub-problem is assumed to have non-negativity constraints only. On the contrary, our framework can handle any convex constraint set $\Zcal_i$ for which the proximal operator can be easily computed, that is 
\begin{align}
\prox^{\iota_\Zcal}_\alpha\left(x\right):=\argm_y\iota_\Zcal\left(y\right)+\frac{\alpha}{2}\left\|y-x\right\|^2_2
\end{align}
is cheap to evaluate. This is the case when the constraint set $\Zcal_i$ is a ball, an ellipsoid, a box, the positive orthant or even second order conic constraints and semidefinite constraints. 
\begin{rk}
There are also many examples  of non-convex constraint sets for which the proximal operator is easily computable, such as mixed integer sets. However, the analysis would not be valid anymore, as the convexity assumption is required in the strong regularity framework \cite{robin1980}.
\end{rk}
By introducing extra variables $y_i$, the nonlinear program \eqref{eq:mu_conv_pb} can be rewritten
\begin{align}
\label{eq:mu_conv_pb_2}
&\minimise~f(y_1,\ldots,y_P,s_t)\\
&\text{s.t.}~~~~g(y_1,\ldots,y_P,s_t)=0\nonumber\\
&~~~~y_i-z_i=0,~\forall i\in\left\{1,\ldots,P\right\}\nonumber\\
&~~~~z_i\in\Zcal_i,~\forall i\in\left\{1,\ldots,P\right\}\enspace.\nonumber
\end{align}
As a result, after defining 
\begin{align}
&S_{\rho}\big(y_1,\ldots,y_P,z_1,\ldots,z_P,\mu,\nu_1,\ldots,\nu_P,s\big):=\nonumber\\
&~~~~f\big(y_1,\ldots,y_P,s\big)+\mu^{\Trans}g\big(y_1,\ldots,y_P,s\big)\nonumber\\
&+\displaystyle\frac{\rho}{2}\left\|g\big(y_1,\ldots,y_P,s\big)\right\|_2^2+\sum_{i=1}^m\nu_i^{\Trans}\big(y_i-z_i\big)+\displaystyle\frac{\rho}{2}\left\|y_i-z_i\right\|_2^2\enspace,
\end{align}
the primal alternations of Algorithm \ref{algo:mu_conv_al} consist in two sorts of steps:
\begin{align}
\label{eq:lin_sys_stp}
\minimise_{y_i\in\Rset^{n_i}}S_{\rho}\big(&y_1^{(l+1)},\ldots,y_{i-1}^{(l+1)},y_i,y_{i+1}^{(l)},\ldots,y_P^{(l)},\nonumber\\
                                                                         &z_1^{(l)},\ldots,z_P^{(l)},\mu,\nu_1,\ldots,\nu_P\big)+\displaystyle\frac{\alpha_i}{2}\left\|y_i-y_i^{(l)}\right\|_2^2\enspace,
\end{align}
which is an unconstrained QP, if $f$ is quadratic, and can therefore be solved in closed-form, and
\begin{align}
\minimise_{z_i\in\Zcal_i}\nu_i^{\Trans}\left(y_i^{(l+1)}-z_i\right)&+\displaystyle\frac{\rho}{2}\left\|y_i^{(l+1)}-z_i\right\|_2^2\nonumber\\
&~~~~~~~+\displaystyle\frac{\alpha_i}{2}\left\|z_i-z_i^{(l)}\right\|_2^2\enspace,
\end{align}
which can be rewritten
\begin{align}
\label{eq:proj_stp}
\minimise_{z_i\in\Zcal_i}\left\|z_i-\displaystyle\frac{1}{\alpha_i+\rho}\left(\alpha_iz_i^{(l)}+{\rho}y_i^{(l+1)}+\nu_i\right)\right\|_2\enspace,
\end{align}
and thus corresponds to projecting 
\begin{align}
\displaystyle\frac{1}{\alpha_i+\rho}\left(\alpha_iz_i^{(l)}+{\rho}y_i^{(l+1)}+\nu_i\right)
\end{align} 
onto $\Zcal_i$. The solution can be obtained in closed-form in many cases, as previously mentioned.
\subsection{A real-time NMPC scheme for bilinear models}
\label{subsec:nmpc_bilin}
Bilinear models encompass a large variety of physical processes \cite{bruni1974} and allow one to capture phenomena, which would be difficult to represent via linear models, while remaining relatively simple. For instance, many examples in Power Systems are bilinear systems, in which the control variable has a multiplicative effect on the state \cite{moh1996}.\\
Therefore, we consider discrete-time constrained bilinear models in the following form:
\begin{align}
\label{eq:bilin_mod_dis}
&x_{l+1}=Ax_l+Bu_l+\sum_{i=1}^{m}u_l^{(i)}N_ix_l\nonumber\\
&\underline{x}\leq x_l\leq\overline{x},~\underline{u}\leq u_l\leq\overline{u}\enspace,
\end{align}
where $x_l\in\Rset^n$ and $u_l\in\Rset^m$. An NMPC problem for \eqref{eq:bilin_mod_dis} can be formalised as
\begin{align}
\label{eq:mpc_bilin}
&\minimise~\sum_{l=0}^{N-1}L(x_k,u_k)+L_f(x_N)\\
&\text{s.t.}~x_0=\hat{x}_0\enspace,\nonumber\\
&~~~~x_{l+1}=Ax_l+Bu_l+\sum_{i=1}^mu_l^{(i)}N_ix_l\enspace,\nonumber\\
&~~~~\underline{x}\leq x_l\leq\overline{x},~\underline{u}\leq u_l\leq\overline{u},~l\in\left\{0,\ldots,N-1\right\},\nonumber\\
&~~~~x_N\in\Xcal\enspace,\nonumber
\end{align}
where $L\left(\cdot,\cdot\right)$ is a quadratic stage-cost, $L_f\left(\cdot\right)$ is an appropriate (convex) terminal weight and the terminal constraint set $\Xcal$ is assumed to be a box. By introducing extra-variables the two main steps of Algorithm \ref{algo:mu_conv_al} are similar to \eqref{eq:lin_sys_stp}, which corresponds to linear system solving, and \eqref{eq:proj_stp}, which consists in clipping onto boxes in the case of \eqref{eq:mpc_bilin}. Thus, in the case of bilinear systems, the real-time implementation of Algorithm \ref{algo:mu_conv_al} consists in a fixed amount of very simple operations, which can be easily parallelised.
%%%%%%%%%%%%%%%%%%%%%%%%%%%%%%%%%%%%%%%%%%%%%%%%%%%%%%%%%%%%%%%%%%%%%%%%%%%%%%%%%%%%%%%%%%%%%%%%%%%%%%%%%%%%%%%%%%%%%%%%%%%%%
\section{Numerical example}
\label{sec:num_ex}
The efficacy of Algorithm \ref{algo:mu_conv_al} is demonstrated at controlling a simple bilinear system, namely a DC motor. The discrete-time dynamics are 
\begin{align}
x_{l+1}=A_dx_l+B_dx_lu_l+c_d\enspace,
\end{align}
where 
\begin{align}
&A_d:=\begin{pmatrix}
1-\frac{R_a{\Delta}t}{L_a} & 0 \\
0 & 1-\frac{B{\Delta}t}{J}
\end{pmatrix},~B_d:=\begin{pmatrix}
0 & -\frac{k_m{\Delta}t}{L_a} \\
\frac{k_m{\Delta}t}{J} & 0
\end{pmatrix}\enspace,\nonumber\\
&~~~~~~~~~~~~~~~~~~~~~~~~c_d:={\Delta}t\begin{pmatrix}
\frac{u_a}{L_a}\\ -\frac{\tau_l}{J} 
\end{pmatrix}\enspace,
\end{align}
with ${\Delta}t$ the sampling period and the parameters values, taken from \cite{dan1998}:
\begin{align}
&L_a=0.307~\text{H},~R_a=12.548~\Omega,~k_m=0.22567~\nicefrac[]{\text{Nm}}{\text{A}^2}\enspace,\nonumber\\
&J=0.00385~\text{Nm.sec}^2,~B=0.00783~\text{Nm.sec}\enspace,\nonumber\\
&\tau_l=1.47~\text{Nm},~u_a=60~\text{V}\enspace.
\end{align}
In the state variable, $x_k(1)$ is the armature current, while $x_k(2)$ is the angular speed. The control input is the field current of the machine. The control objective is to make the angular speed track a piecewise constant reference $\pm2~\nicefrac[\text]{rad}{sec}$, while satisfying the following state and input constraints:
\begin{align}
&\underline{x}=\begin{pmatrix}
-2~\text{A}\\ -8~\text{rad/sec}
\end{pmatrix},~\overline{x}=\begin{pmatrix}
5~\text{A}\\ 1.5~\text{rad/sec}
\end{pmatrix}\enspace,\nonumber\\
&\underline{u}=1.27~\text{A},~\overline{u}=1.4~\text{A}\enspace.
\end{align}
The NMPC problem \eqref{eq:mpc_bilin} is solved via Algorithm \ref{algo:mu_conv_al}. In order to assess its performance, our tracking algorithm is tested for different sampling periods, while initialised at a perturbed solution $5\cdot w^{\ast}_0$, where the primal-dual optimal solution $w^\ast_0$ has been computed using \textsc{ipopt} \cite{waech2006}.
\begin{figure}[h!]
\begin{center}
\includegraphics[scale=0.12]{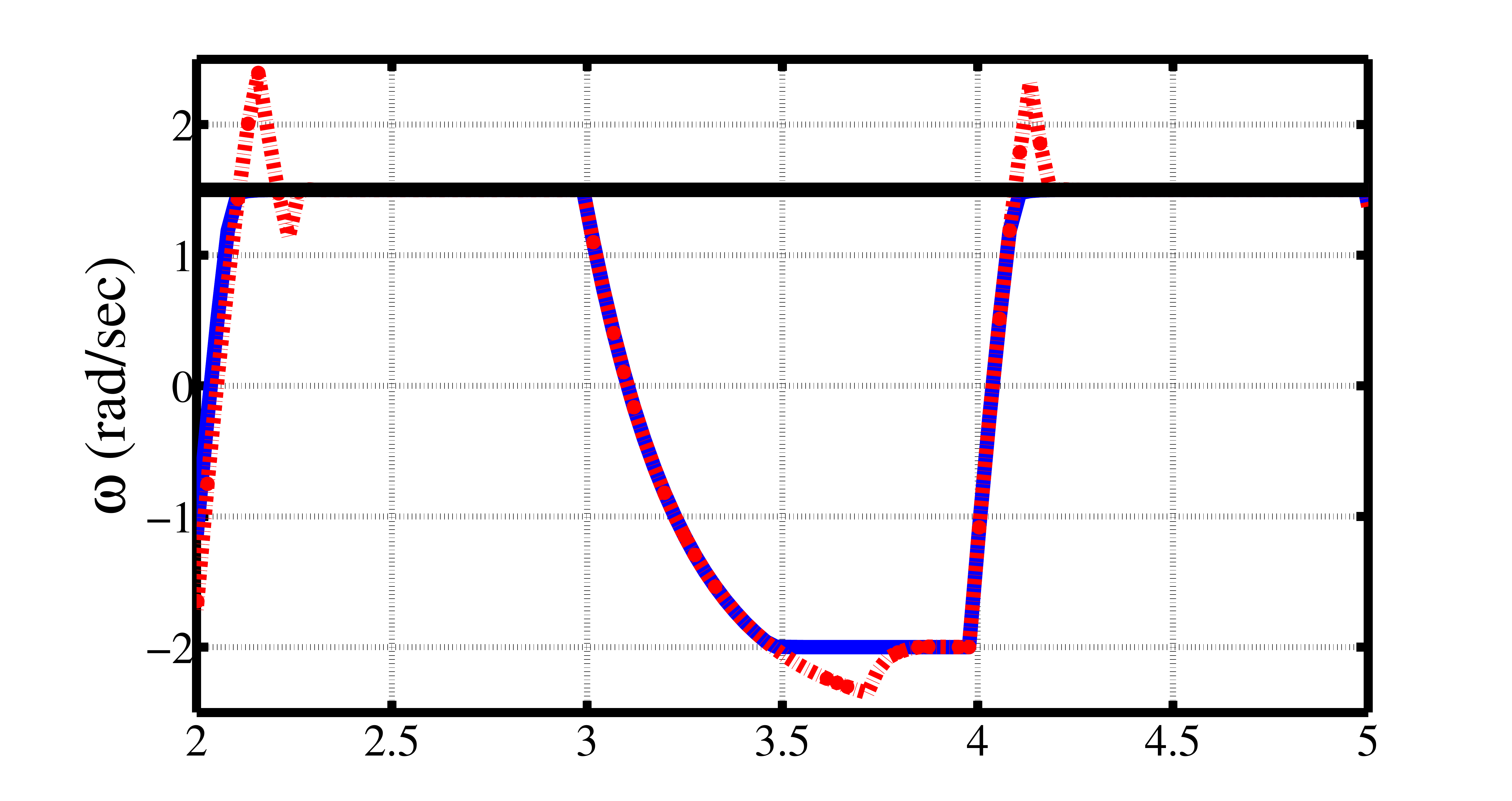}\\ 
\vspace{-0.2cm}
\includegraphics[scale=0.12]{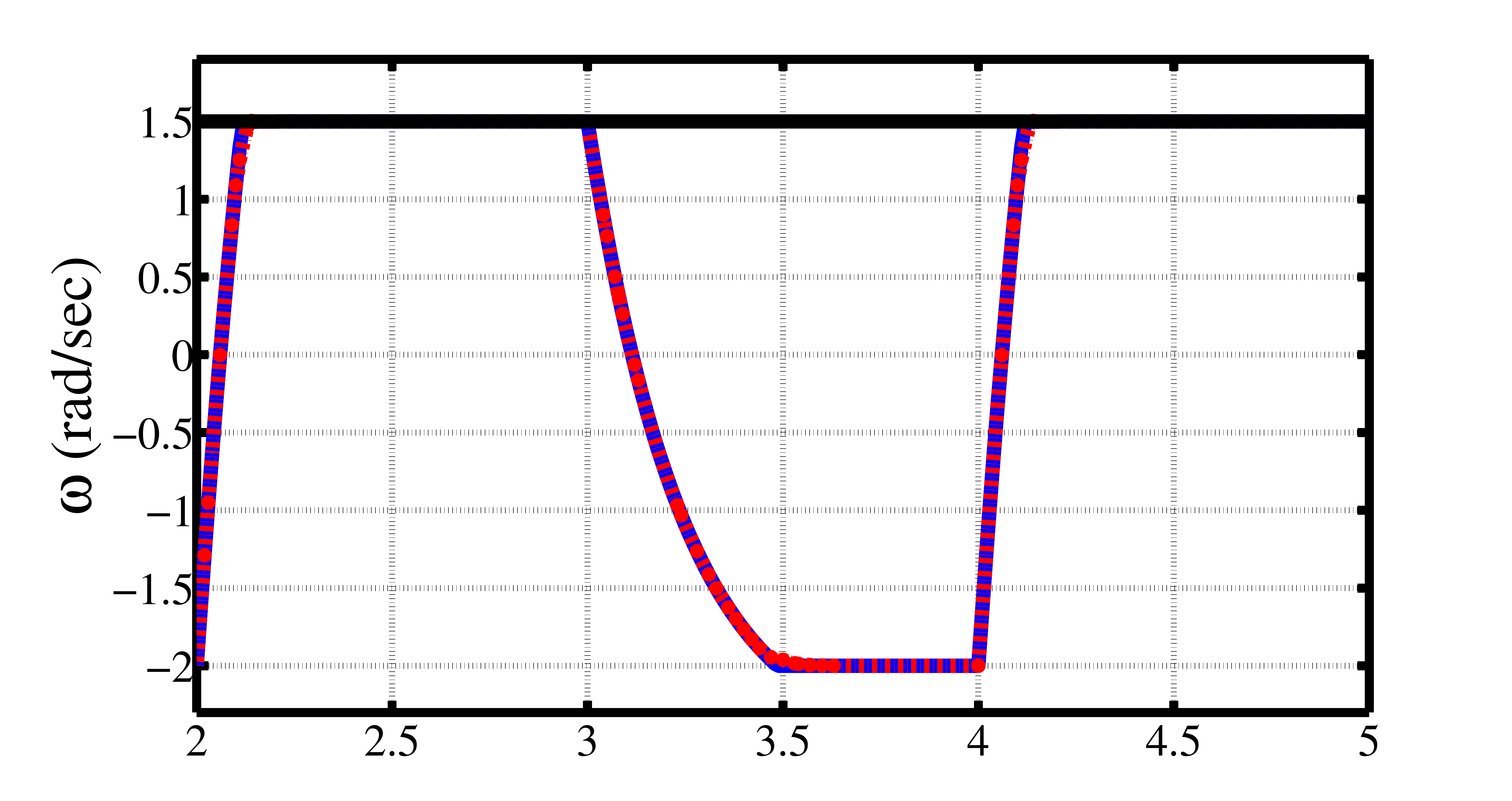}
\put(-110,-5){\footnotesize Time (s)}
\end{center}
\caption{\label{fig:speed_tracks}Speed responses for ${\Delta}t=0.026~\text{sec}$ (top) and ${\Delta}t=0.01~\text{sec}$ (bottom): full NMPC solved using \textsc{ipopt} in blue, using Algorithm \ref{algo:mu_conv_al} in dashed red.}
\end{figure}

The speed trajectories are plotted in Fig. \ref{fig:speed_tracks} and the input in Fig. \ref{fig:inputs}. It clearly appears that as the sampling period is low, the tracking performance is better, the full NMPC trajectory and the sub-optimal one are almost the same. For a larger sampling period, the state constraints may be violated, as illustrated in Fig. \ref{fig:speed_tracks}, while the input constraints are always satisfied, as shown on Fig. \ref{fig:inputs}, due to the formulation of Algorithm \ref{algo:mu_conv_al}. 
\begin{figure}[h!]
\begin{center}
\includegraphics[scale=0.12]{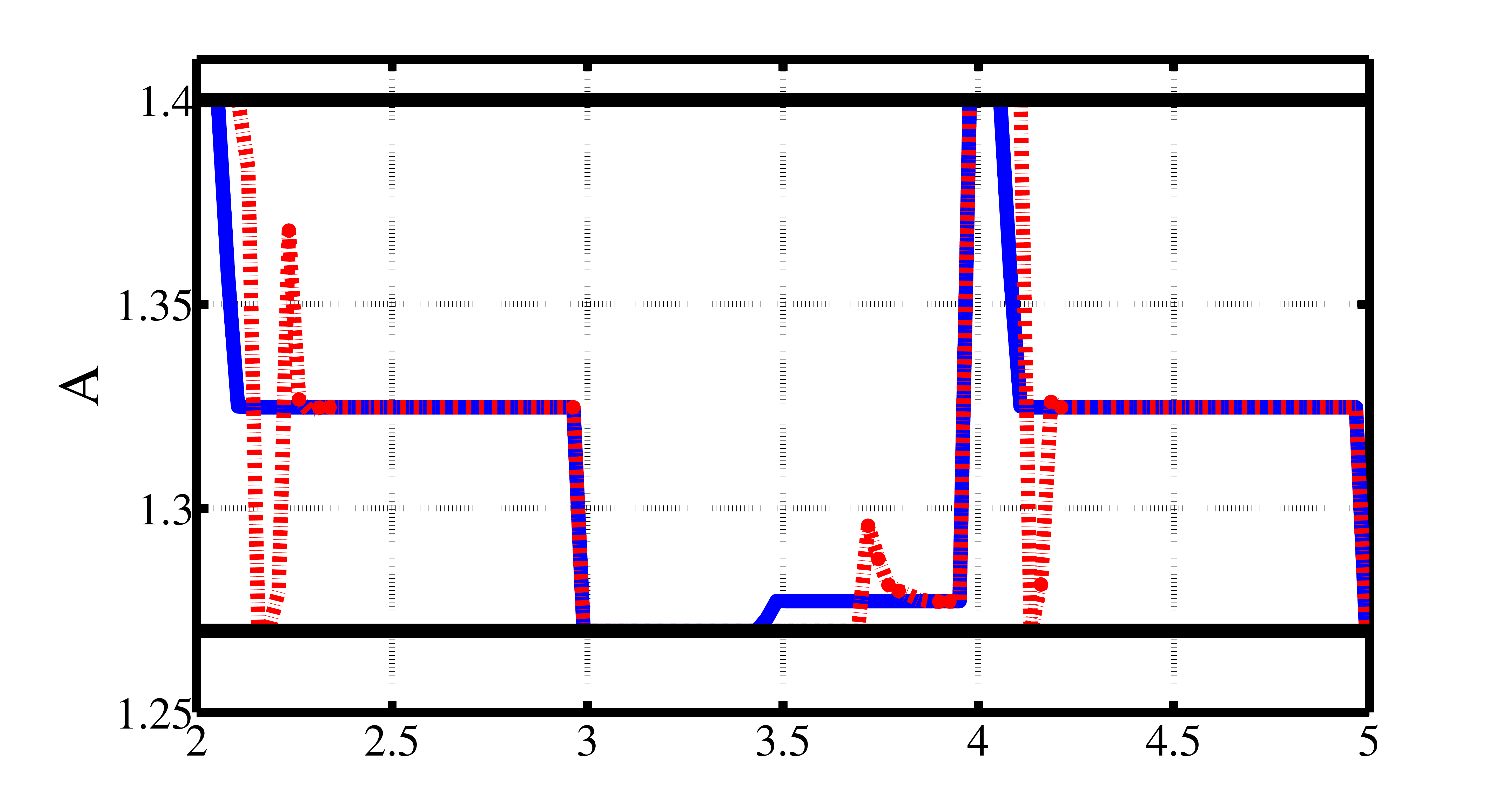}\\ 
\vspace{-0.2cm}
\includegraphics[scale=0.12]{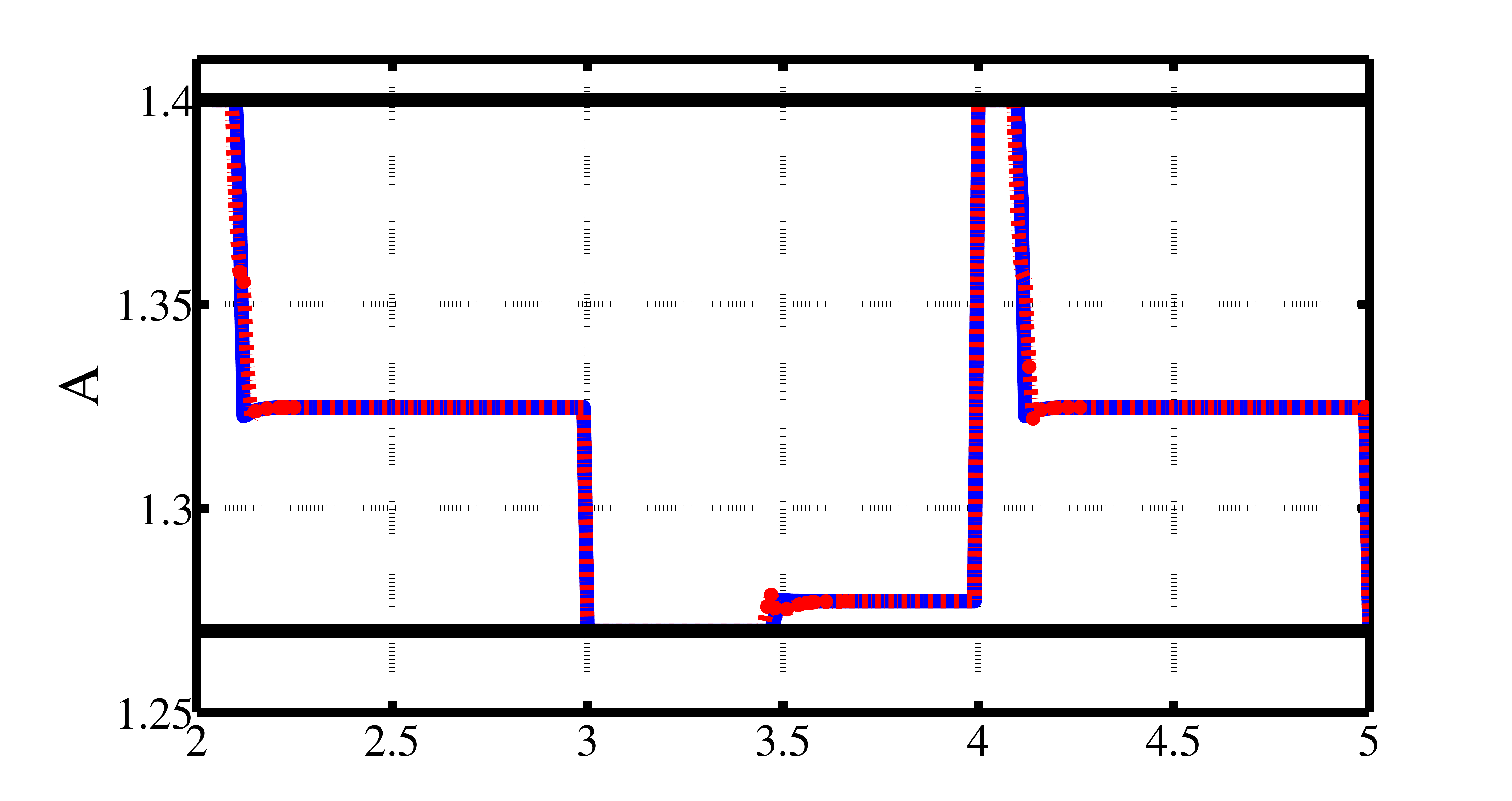}
\put(-110,-5){\footnotesize Time (s)}
\end{center}
\caption{\label{fig:inputs}Input for ${\Delta}t=0.026~\text{sec}$ (top) and ${\Delta}t=0.01~\text{sec}$ (bottom): full NMPC solved using \textsc{ipopt} in blue, using Algorithm \ref{algo:mu_conv_al} in dashed red.}
\end{figure}

The tracking algorithm converges to a feasible solution at a speed depending on the sampling period, as shown in Fig. \ref{fig:feas}. Given a fixed sampling period, increasing the penalty parameter $\rho$ may improve the tracking performance, as a larger penalty mitigates the effect of the error on the dual variables.  

\begin{figure}[h!]
\begin{center}
\includegraphics[scale=0.12]{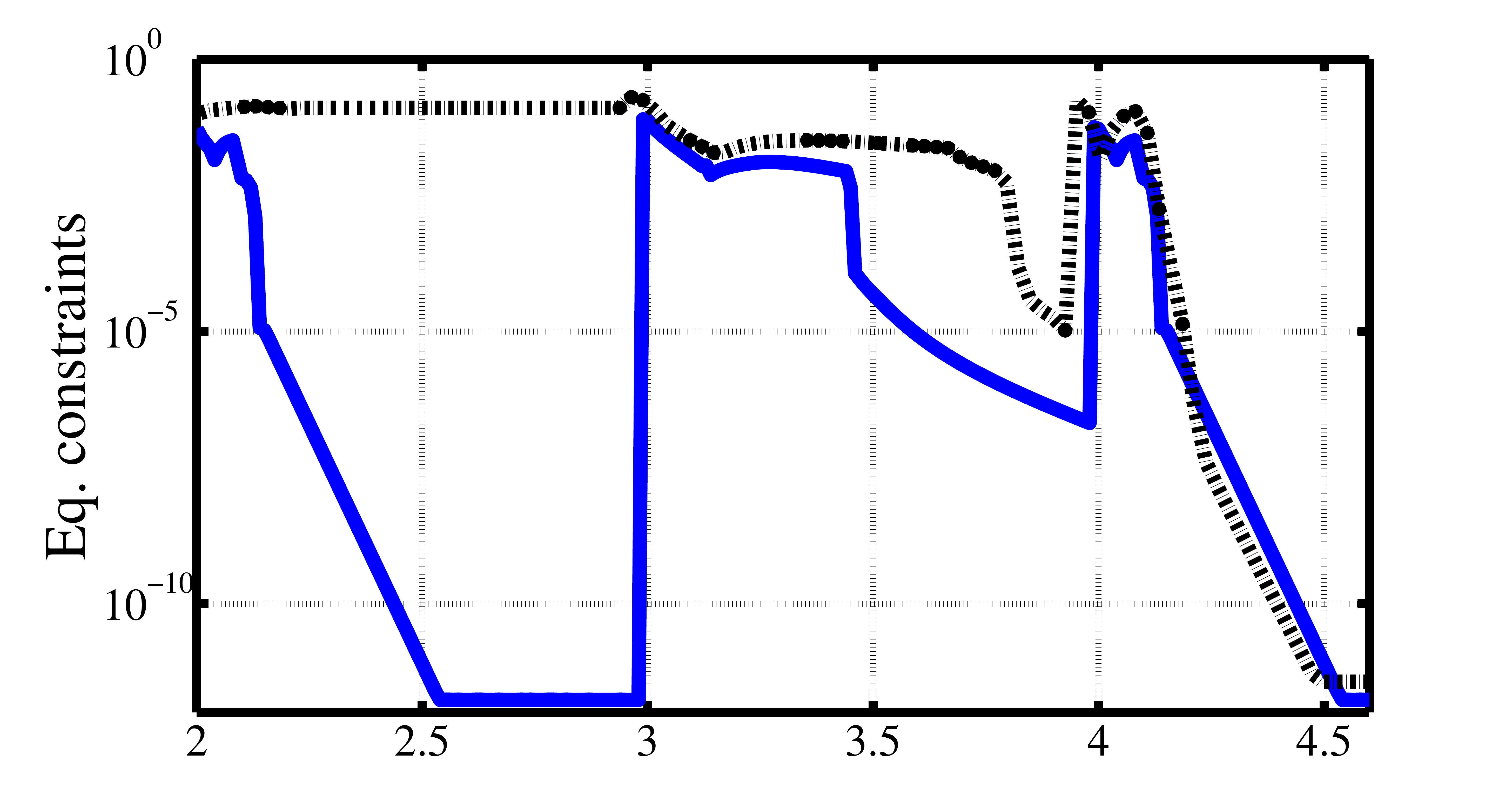}
\put(-110,-5){\footnotesize Time (s)}
\end{center}
\caption{\label{fig:feas}Feasibility of bilinear equality constraints, for ${\Delta}t=0.01~\text{sec}$ in blue and ${\Delta}t=0.026~\text{sec}$ in dashed black.}
\end{figure}
Finally, the computational power is fixed artificially, that is a maximum number of iterations per second is given a priori. Then the sampling period is made vary within a fixed range and the performance of Algorithm \ref{algo:mu_conv_al} is measured using the normalised L$2$-norm of the difference between the full NMPC trajectory and the sub-optimal one obtained by tracking at the given time period. 
\begin{figure}[h!]
\begin{center}
\includegraphics[scale=0.1]{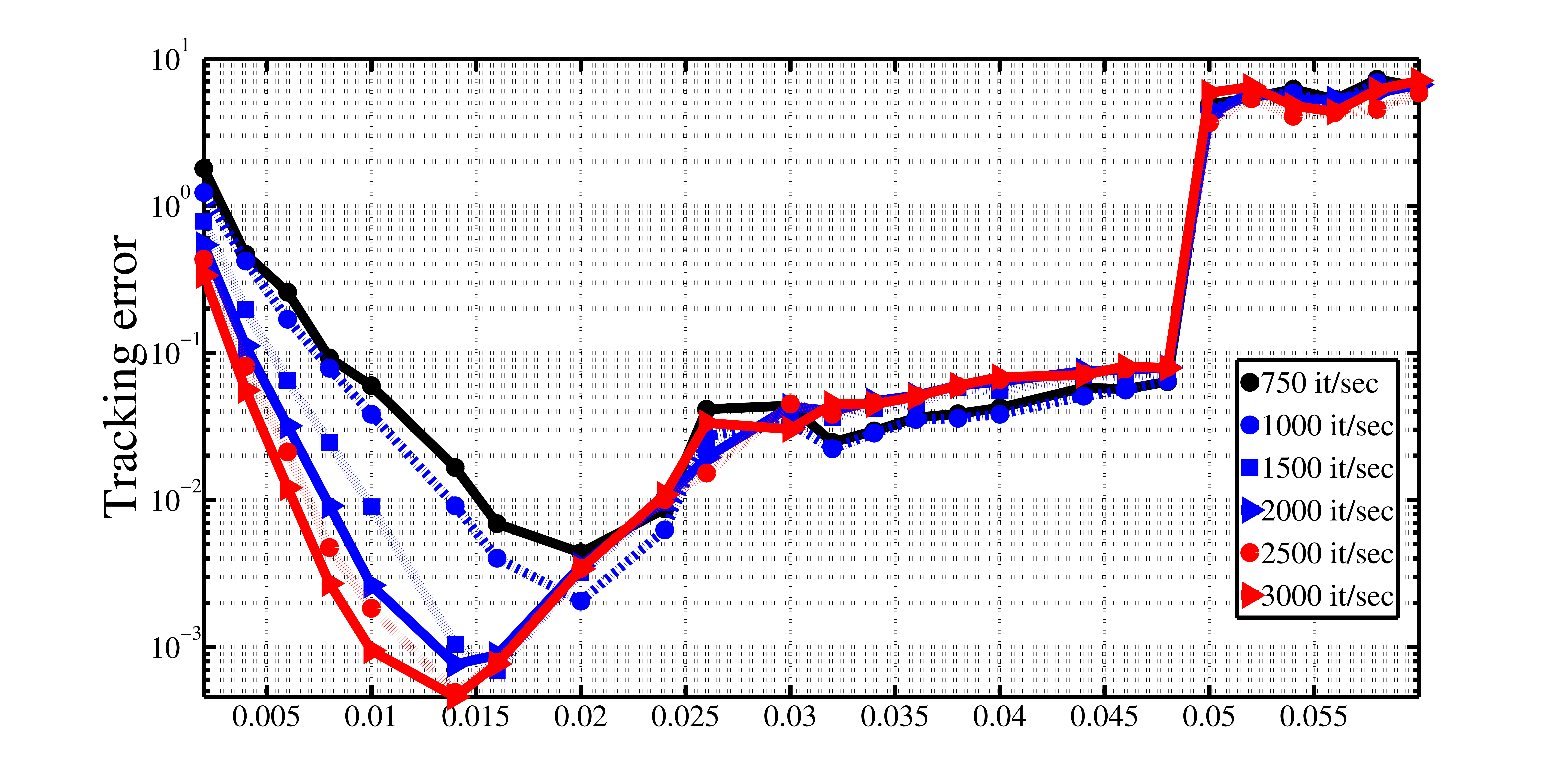}
%\includegraphics[scale=0.1]{/Users/jean/Desktop/Publications/Conf_papers/CDC_2014/paretosGlob.pdf}
%\vspace{-0.2cm}
%\includegraphics[scale=0.1]{/Users/jean/Desktop/Publications/Conf_papers/CDC_2014/paretosL2_30.pdf}
\put(-140,-5){\footnotesize$\Delta t$~(s)}
\end{center}
\caption{\label{fig:paretos}Evolution of the tracking error (normalised L$2$-norm) versus sampling period for different computational powers.}
\end{figure}
As the sampling period increases, more iterations are allowed, so the tracking error decreases, as pictured on Fig. \ref{fig:paretos}. If the sampling period is too large, the warm-start is too far from the optimal solution and increasing the number of iterations cannot help reducing the error, as only one dual update is performed at each time step. As a result, the tracking error explodes for large sampling periods.   
%%%%%%%%%%%%%%%%%%%%%%%%%%%%%%%%%%%%%%%%%%%%%%%%%%%%%%%%%%%%%%%%%%%%%%%%%%%%%%%%%%%%%%%%%%%%%%%%%%%%%%%%%%%%%%%%%%%%%%%%%%%%%
\section{Conclusion}
\label{sec:con}
A parametric splitting technique has been presented in order to solve time-dependent multi-convex parametric problems. A contraction estimate has been derived, which guarantees boundedness of the error sequence assuming the parameter difference is small enough. Finally, efficacy of our approach has been assessed on a realistic example consisting in speed control of a DC motor using NMPC. Our algorithm seems to be well-adapted to parallel computational environments and can be further extended to solve distributed NMPC problems in a real-time framework.
\bibliographystyle{plain}	
\bibliography{biblio}	

\begin{thebibliography}{10}

\bibitem{att2009}
H.~Attouch and J.~Bolte.
\newblock On the convergence of the proximal algorithm for nonsmooth functions
  involving analytic features.
\newblock {\em Mathematical Programming}, $116$:$5$--$16$, $2009$.

\bibitem{att2010}
H.~Attouch, J.~Bolte, P.~Redont, and A.~Soubeyran.
\newblock Proximal alternating minimisation and projection methods for
  non-convex problems: an approach based on the {K}urdyka-{L}ojasiewicz
  inequality.
\newblock {\em Mathematics of Operations Research}, $35$:$438$--$457$, $2010$.

\bibitem{bert1982}
D.P. Bertsekas.
\newblock {\em {C}onstrained optimisation and {L}agrange multiplier methods}.
\newblock {A}thena {S}cientific, $1982$.

\bibitem{bert1997}
D.P. Bertsekas and J.N. Tsitsiklis.
\newblock {\em Parallel and distributed computation: numerical methods}.
\newblock {A}thena {S}cientific, $1997$.

\bibitem{bolte2007}
J.~Bolte, A.~Daniilidis, and A.~Lewis.
\newblock The {L}ojasiewicz inequality for nonsmooth sub-analytic functions
  with applications to subgradient dynamical systems.
\newblock {\em SIAM Journal on Optimisation}, $17$:$1205$--$1223$, $2007$.

\bibitem{boyd2010}
S.~Boyd, N.~Parikh, E.~Chu, B.~Peleato, and J.~Eckstein.
\newblock Distributed optimisation and statistical learning via the alternating
  direction method of multipliers.
\newblock {\em Foundations and Trends in Machine Learning},
  $3$($1$):$1$--$122$, $2010$.

\bibitem{bruni1974}
C.~Bruni, G.~Di~Pillo, and G.~Koch.
\newblock Bilinear systems: an appealing class of `nearly linear' systems in
  theory and applications.
\newblock {\em IEEE Transactions on Automatic Control}, $19$($4$):$334$--$348$,
  $1974$.

\bibitem{conn1996}
A.R. Conn, N.~Could, A.~Sartenaer, and P.L. Toint.
\newblock Convergence properties of an augmented {L}agrangian algorithm for
  optimisation with a combination of general equality and linear constraints.
\newblock {\em SIAM Journal on Optimization}, $6$($3$):$674$--$703$, $1996$.

\bibitem{acunto2005}
D.~D'Acunto and K.~Kurdyka.
\newblock Explicit bounds for the {L}ojasiewicz exponent in the gradient
  inequality for polynomials.
\newblock {\em {A}nnales {P}olonici {M}athematici}, $87$, $2005$.

\bibitem{dan1998}
S.~Daniel-Berhe and H.~Unbehauen.
\newblock Experimental physical parameter estimation of a thyristor driven
  {DC}-motor using the {HMF}-method.
\newblock {\em Control Engineering Practice}, $6$:$615$--$626$, $1998$.

\bibitem{diehl2005}
M.~Diehl, H.G. Bock, and J.P. Schloeder.
\newblock A real-time iteration scheme for nonlinear optimisation in optimal
  feedback control.
\newblock {\em SIAM Journal on Control and Optimisation},
  $43$($5$):$1714$--$1736$, $2005$.

\bibitem{diehl2008}
H.J. Ferreau, H.G. Bock, and M.~Diehl.
\newblock {\em International Journal of Robust and Nonlinear Control},
  $18$($8$):$816$--$830$, $2008$.

\bibitem{fiac1976}
A.V. Fiacco.
\newblock Sensitivity analysis for nonlinear programming using penalty methods.
\newblock {\em Mathematical Programming}, $10$:$287$--$311$.

\bibitem{loja1963}
S.~Lojasiewicz.
\newblock Une propri\'et\'e topologique des sous-ensembles analytiques r\'eels.
\newblock In {\em Les {\'E}quations aux {D}\'eriv\'ees {P}artielles}, pages
  $87$--$89$. {\"E}ditions du CNRS, $1963$.

\bibitem{moh1996}
R.~Mohler and R.~Zakrzewski.
\newblock Nonlinear control algorithms and power system application.
\newblock {\em Applied Mathematics and Computation}, $78$:$197$--$207$, $1996$.

\bibitem{robin1980}
Stephen~M. Robinson.
\newblock Strongly regular generalised equations.
\newblock {\em Mathematics of Operations Research}, $5$($1$):$43$--$62$,
  $1980$.

\bibitem{rock2009}
R.T. Rockafellar and R.~J.-B. Wets.
\newblock {\em Variational analysis}.
\newblock Springer, $2009$.

\bibitem{waech2006}
A.~Waechter and L.T. Biegler.
\newblock On the implementation of a primal-dual interior point filter
  line-search algorithm for large-scale nonlinear programming.
\newblock {\em Mathematical Programming}, $106$($1$):$25$--$57$, $2006$.

\bibitem{zav2010}
V.M. Zavala and M.~Anitescu.
\newblock Real-time nonlinear optimisation as a generalised equation.
\newblock {\em SIAM Journal on {C}ontrol and {O}ptimization},
  $48$($8$):$5444$--$5467$, $2010$.

\bibitem{zav2013}
V.M. Zavala and M.~Anitescu.
\newblock Scalable dynamic optimization.
\newblock SIAM Conference on Computational Science and Engineering, $2013$.

\bibitem{zav2009}
V.M. Zavala and L.T. Biegler.
\newblock The advanced-step {NMPC} controller: optimality, stability and
  robustness.
\newblock {\em Automatica}, $45$($1$):$86$--$93$, $2009$.

\end{thebibliography}
\end{document}